\documentclass[
a4paper,
eqnright, reqno]{amsart}
\usepackage{graphicx}
\usepackage{myamsart}

\providecommand{\Im}{\loglike{lm}}

\newtheorem{mydef}{Definition}

\newtheorem{Problem}{Problem}
\title[Constructive Method for Scalar
  Wiener-Hopf Equations]{A Constructive Method for Approximate Solution to Scalar
  Wiener-Hopf Equations}
\author{Anastasia V. Kisil}    
\address{Faculty of Mathematics, University of Cambridge,
  Wilberforce Road, Cambridge, CB3 0WA, UK }
\email{a.kisil@maths.cam.ac.uk}
 
\begin{document}

\begin{abstract}
This paper presents a novel method of approximating the scalar Wiener-Hopf
equation; and therefore constructing an approximate solution. The
advantages of this method over the existing methods are
reliability and explicit error bounds.  Additionally the degrees of the
polynomials in the rational approximation are considerably smaller than in other
approaches. 

The need for a numerical solution is motivated by difficulties in
computation of the exact solution. The approximation developed in this
paper is with a view of generalisation to matrix Wiener-Hopf for which
the exact solution, in general, is not known.

The first part of the paper develops error bounds in \(L_p\) for \(1<p<
\infty \). These indicate how accurately the solution is approximated
in  terms of how accurate the equation is approximated. 
 
The second part of the paper describes the approach of approximately solving the Wiener-Hopf
equation that employs the  Rational Carath\'{e}odory-Fej\'{e}r Approximation. The
method is adapted by constructing a mapping of the real line to the unit
interval. Numerical examples to demonstrate the use of the
proposed technique are included (performed on Chebfun), yielding error
as small as \(10^{-12}\) on the whole real line.  
\end{abstract} 


\maketitle

\section{Introduction}

 The Wiener-Hopf method is used for  
 a broad collection of PDEs which arise in acoustic, finance,
 hydrodynamic, elasticity, potential and electromagnetic theories
 \cite{Mishuris09, Ab_ex}. It is an elegant 
 method based on the exploitation of the analyticity properties of the
 functions. For the scalar Wiener-Hopf the solution can
 be expressed in terms of a Cauchy type integral  \cite{bookWH}*{Ch. 1.3}.

 In more complicated scalar Wiener-Hopf equations the exact solution
 is difficult or slow to compute, see e.g.  \cite{diffucultie_1,
   diffucultie_2, bounds}.  Approximate solutions were considered
 early on but they were mainly constructed using \emph{ ad hoc}
 observations \cite{bookWH}*{Ch. 4.5}. In 2000 a systematic way of
 approximating the Wiener-Hopf equations was published by I. D.
 Abrahams \cite{Pade}. Since then it proved popular and found
 applications in different branches of mathematics including finance
 \cite{finance}.

 The method proposed in \cite{Pade} is based on uniform approximations
 of the kernel on the whole strip by a two point Pad\'{e}
 approximation (the two points being 0 and \(\infty\) on the real
 axis). However, there are two issues that make the application of
 this method difficult. The first is that it is unclear when the two
 point Pad\'{e} approximation has small error on the whole strip
 \cite{Pade_book}. Secondly, even if the maximum error of
 approximating the kernel on the whole strip is known, there are no
 error bounds presented to calculate the resulting error in the
 solution.

Another motivation for the development of approximate methods is the
matrix Wiener-Hopf problem. The determination of a good numerical solution is
important for the matrix Wiener-Hopf problem, since as yet there is no
constructive way of solving it in general. 

This paper aims to present a consistent method of
approximately solving the scalar  Wiener-Hopf. The following questions
are addressed:

\begin{Problem}
Given a scalar Wiener-Hopf problem find an approximate solution which can be
demonstrated to be within a given accuracy from the exact solution. 
\end{Problem}

\begin{Problem}
How to perform Problem \(1\) computationally in a way that is reliable, optimal
and easy.
\end{Problem}

Problem~\(1\) will be addressed in Section~\(3\) with new
estimations for the error in the Wiener-Hopf factors. This is done
though expressing the factors in terms of the Hilbert
transform, which is connected to the Cauchy singular integral. Note
that although the method proposed in this paper involves the
construction of an approximate solution to the given Wiener-Hopf
equation, the solution is the exact solution of a perturbed
Wiener-Hopf equation.

The main difficulty in Problem~\(2\) is that the Wiener-Hopf equation
is set on an unbounded interval. We propose the novel method based on
an appropriate mapping to the unit interval, see Section~\(4\). This
allows the Carath\'{e}odory-Fej\'{e}r rational approximation to then
be used in Section~\(5\). Lastly, numerical examples are given.

\section{Preliminaries}
\label{sec:Preliminaries}

The following conventions will be used throughout the paper. \emph{A strip} around the real axis (given \(\tau_{-}<0 <\tau_{+}\))  is \{\(\alpha=\sigma+i\tau\) :
\(\tau_{-}<\tau<\tau_{+}\)\}\(\subset \mathbb{C}\). 
 The subscript \(+\) (or \(-\)) indicates  that the
function is \emph{analytic  in the half-plane} \(\tau>\tau_{-}\) (or \(\tau<\tau_{+}\)). 
 Functions in the Wiener-Hopf equation without the subscript are
  analytic in the strip. 

The Wiener-Hopf problems are recalled below.

The  \emph{multiplicative Wiener-Hopf} problem is: given a function \(K\)
(analytic, zero-free
 and \(K(\alpha) \to 1\) as \( |\sigma| \to \infty\) in the strip) to find functions \(K_+\) and \(K_-\) which satisfy the
following equation in the strip:
\begin{equation}
  \label{eq:mul}
K(\alpha)= K_-(\alpha)K_+(\alpha).
\end{equation}
In addition \(K_-\) and \(K_+\) are required to be:
\begin{itemize}
\item analytic and  non-zero in the respective half-plane. 
\item of \emph{subexponential growth} the respective
  half-planes:
\[|\log K_\pm(\alpha)|=O(|\alpha|^p), \quad p<1,  \quad \text{as} \quad
|\sigma| \to \infty.\]
\end{itemize}
The function \(K(\alpha)\)  in the multiplicative Wiener-Hopf
factorisation will be referred to as  \emph{the kernel}. The functions \(K_+\) and \(K_-\)
will be called \emph{factors} of \(K\). Those factor are unique up to
a constant. In other words if there are two such factorisations
\(K=K_+K_-\) and \(K=P_+P_-\) then:
\[K_+=cP_+ \quad \text{and} \quad K_-=c^{-1}P_-,\]
where \(c\) is an arbitrary complex number \cite{Kranzer_68}.  In this paper they will be
normalised so that \(K_\pm(\alpha) \to 1\) as \( |\sigma| \to \infty\) in the strip.

\begin{rem}
  The above normalisation at infinity allows to control the regularity
  of the Fourier transform of an approximation. This is useful if
  a factorisation of the Fourier transform is used, say, for a solution
  of differential equations.
\end{rem}

The multiplicative Wiener-Hopf factorisation can be reduced to the
\emph{additive Wiener-Hopf} problem via application of logarithm. Under the conditions on the
kernel \(K\), let \(f=\log K\) then additive Wiener-Hopf problem is:
given a function  \(f\) to find functions \(f_+\) and \(f_-\) which
satisfy the  following equation in the strip:
\begin{equation}
\label{eq:add}
f(\alpha)= f_-(\alpha)+f_+(\alpha).
\end{equation}
Additive and multiplicative problems
are equivalent only in the scalar case. In the matrix Wiener-Hopf
\(K_+\) and \(K_-\) will in general not be  commutative so no such
equivalence is possible.

The existence of solutions to the additive and multiplicative
scalar Wiener-Hopf problem is addressed in for example \cite{bookWH}*{Ch. 1.3}.

These two Wiener-Hopf splittings are  key to solving the \emph{general
  Wiener-Hopf problem}. That is:  given functions \(A\) and \(C\) find functions \(\Phi_+\) and \(\Psi_-\), which satisfy the 
following equation in the strip:
\begin{equation}
  \label{eq:gen}
A( \alpha)\Phi_+(\alpha)+ \Psi_-(\alpha)+C(\alpha)=0.
\end{equation}
For more details about the Wiener-Hopf problem see  ~\cite{bookWH}.

\subsection{The Riemann-Hilbert Problem}

The Wiener-Hopf problem is a special case of the Riemann-Hilbert
problem. Roughly, the Riemann-Hilbert
problem has a more general contour instead of the strip and
the conditions on the function \(A(\alpha)\) are weakened. In
particular it is required that   \(A(\alpha)\) is  only  H\"{o}lder
continuous on the contour. 

Let   \( G(t), g(t)\) be H\"{o}lder continuous functions on
a simple contour \(\Sigma\). The \emph{Riemann-Hilbert problem} is to find an  analytic  function
\(E(\alpha)\) which has
 values \(E_-(t), E_+(t)\) on the contour as the limit is taken
from different sides of the contour and which satisfies the equation:
\[G( t)E_-(t)+ E_+(t)+g(t)=0, \qquad t \in \Sigma.\] 
The Wiener-Hopf equation \eqref{eq:gen} can be considered as the special case of the 
Riemann-Hilbert problem  so a good 
approximation of the values of the function on the real 
line is sufficient. This is simpler  than trying to construct an 
approximation that agrees well on the whole strip of analyticity.

\section{Estimates on the Approximate Wiener-Hopf Factorisation}
\label{sec:bound}

The approximate solution to \eqref{eq:mul} can be found by 
 approximating the kernel in the
Wiener-Hopf equation. This section quantifies the difference between
the modified and the original Wiener-Hopf equation, to address
Problem~\(1\) introduced earlier. More precisely, let  \(|K(\alpha)-\tilde{K}(\alpha)|_{p} \le \epsilon_p\), 
the aim is to  bound \(|K_{\pm}(\alpha)-\tilde{K}_{\pm}(\alpha)|_{p}\) in
terms of  \(\epsilon_p\) and computable quantities of \(K(\alpha)\).

The first step is to link  the Wiener-Hopf
factorisation to the Hilbert Transform \(H(f)(y)\) of \(f\) with
\(y\in \mathbb{R}\). If:
\[f(y)= f_+(y)+  f_-(y),\]
then \cite{Pandey}*{Ch.~2}:
\[i H(f)(y)= f_+(y)-  f_-(y),\]
and hence:
\[f_\pm(y)=\frac{1}{2} f(y)\pm\frac{i}{2} H(f)(y).\]
The advantage of expressing \(f_\pm\) in terms of the Hilbert transform is
the following theorem.
\begin{thm}[Titchmarsh-Riesz] \cite{Pandey}*{p. 92} 
\label{thm:hilbert}
Let \(f \in L_p(\mathbb{R}) \) for some \(1<p
  <\infty\), then:
\[||H(f)||_p \le c(p) ||f||_p,\]
where the best constant is:
\begin{equation}
\label{eq:hilbert}
c(p) = \left\{
 \begin{array}{rl}
  \tan( \pi/(2p)) & \text{if } \quad 1<p \le 2,\\
   \cot(\pi/(2p)) & \text{if } \quad 2 <p < \infty.\\ 
 \end{array} \right.
\end{equation}
\end{thm}
Based on these classical estimations we obtain the following bounds on
the additive Wiener-Hopf factorisation.
\begin{lem} [Additive Bounds in \(L_p\) for \(1<p
  <\infty\)]
\label{lem:addbound}
Let \(f(y)= f_+(y)+  f_-(y)\) and \(\tilde{f}(y)=
\tilde{f}_+(y)+  \tilde{f}_-(y)\) with
\(||f(y)-\tilde{f}(y) ||_p < \epsilon_p\) then:

\[||f_{\pm}(y)-\tilde{f}_{\pm}(y) ||_p \le \frac{1}{2}(1+c(p))
\epsilon_p,\]

where \(c(p)\) is defined as in \eqref{eq:hilbert}.
\end{lem}
\begin{proof}
 Expressing in terms of  the Hilbert transform and using
 Minkowski's inequality:
\begin{eqnarray*}
||f_{+}(y)-\tilde{f}_{+}(y) ||_p &=&||\frac{1}{2}
f(y)+\frac{i}{2} H(f)(y) -\frac{1}{2}
\tilde{f}(y)-\frac{i}{2} H(\tilde{f})(y)||_p \\
&\le&||\frac{1}{2}f(y)-\frac{1}{2}\tilde{f}(y)||_p+||\frac{i}{2}
H(f)(y) -\frac{i}{2} H(\tilde{f})(y)||_p\\
&\le& \frac{1}{2}(1+c(p)) \epsilon_p.
\end{eqnarray*}
Here the linearity property of the Hilbert
transform, \(H(f)+H(g)=H(f+g)\) is used.
\end{proof}

The theorem gives bounds on the error on the real line but in fact
since the functions are analytic, the bounds will hold in the
upper/lower half-planes by the maximum modulus principle.

It will be useful to express the Wiener-Hopf factors in terms of the Hilbert transform;
\[K_\pm(y)= \exp \big(\frac{1}{2}\log K(y) \pm\frac{i}{2} H(\log
K)(y)\big)=K^{1/2}(y) \exp\big(\pm \frac{i}{2} H(\log
K)(y)\big).\]
The  main new result of this section is:

\begin{thm}[Multiplicative Bounds in \(L_p\) for \(1<p
  <\infty\)]
\label{thm:mul}
Let  \(K(y)\) and \(\tilde{K}(y)\) be two
kernels and \(m<||K||_p<M\). If
\(||K(y)-\tilde{K}(y)||_p< \epsilon_p\) then:

\[||K_{\pm}(y)-\tilde{K}_{\pm}(y)||_{p} <
\frac{(M+\epsilon_p)^{1/2}\exp (\frac{c(p)\pi}
{2})}{2(m-\epsilon_p)}(1+c(p)) \epsilon_p,\] 

where \(c(p)\) is defined in \eqref{eq:hilbert}.
\end{thm}


\begin{proof}
The first step is to express the product factorisation
\(K(y)=K_+(y)K_-(y)\) as an additive
factorisation. Taking logarithms:
\[\log K(y)= \log K_+(y) + \log K_-(y).\]
From the assumptions on the kernel \(K(y)\),  \(\log K(y)\) will be a continuous and 
single valued function and hence the additive decomposition is well
defined. Given:
\[||K(y)-\tilde{K}(y)||_p< \epsilon_p.\]

Then using the mean value inequality (and
\(||\tilde{K}(y)||_p>m-\epsilon_p\)):

\[||\log(K(y))-\log(\tilde{K}(y))||_p<
\frac{1}{m-\epsilon_p}\epsilon_p.\]

Apply Lemma \ref{lem:addbound}, which gives:
\[||\log(K_{\pm}(y))-\log(\tilde{K}_{\pm}(y))||_p<\frac{1}{2(m-\epsilon_p)}(1+c(p))\epsilon_p.\]

Lastly, note that \(||\tilde{K}_\pm(y)||_p<(M+\epsilon_p)^{1/2}\exp (\frac{c(p)\pi}
{2})\) since:

\begin{eqnarray*}
||\tilde{K}_\pm(y)||_p&=&||\tilde{K}^{1/2}(y) \exp\big(\pm \frac{i}{2} H(\log
\tilde{K})(y)\big)||_p\\
&\le&(M+\epsilon_p)^{1/2} ||\exp\big(\pm \frac{i}{2} H(\log
\tilde{K})(y)\big)||_p\\
&\le&(M+\epsilon_p)^{1/2} ||\exp\big(\pm \frac{1}{2} \Im (H(\log
\tilde{K})(y))\big)||_p\\
&\le &(M+\epsilon_p)^{1/2} ||\exp\big(\pm \frac{1}{2}  (H(\text{Arg}
(\tilde{K}))(y))\big)||_p\\
&\le &(M+\epsilon_p)^{1/2}\exp (\frac{c(p)\pi}
{2}),
\end{eqnarray*}
and so:
\[||K_{\pm}(y)-\tilde{K}_{\pm}(y)||_{p}<
(M+\epsilon_p)^{1/2}\exp (\frac{c(p)\pi}
{2})||\log(K_{\pm}(y))-\log(\tilde{K}_{\pm}(y))||_p,\]

from which the result follows.
\end{proof}

\begin{rem}[Real Kernels]
In the case when  kernels \(K\) and \(\tilde{K}\) are real valued,
the bounds could be simplified, since then
\(||\tilde{K}_\pm(y)||_p<(M+\epsilon_p)^{1/2}\)
and so:

\[||K_{\pm}(y)-\tilde{K}_{\pm}(y)||_{p} <
\frac{(M+\epsilon_p)^{1/2}}{2(m-\epsilon_p)}(1+c(p)) \epsilon_p.\] 

\end{rem}

\begin{rem}[\(L_\infty\) norm]

The above theorem does not include the case of \(L_\infty\) norm, in
fact no such result is possible for the \(L_\infty\) norm as is
demonstrated by the counter example below.

Consider \( K_1(y)=1\) defined on the real line and

\[  K_2 (y) = \left\{ 
 \begin{array}{lll}
         1+ \epsilon & \mbox{if $0 \le y \le \text{arccot} (n_1)$};\\
        1+ \epsilon - \frac{( y -n_1) \epsilon}{n_2-n_1}& \mbox{if $ \text{arccot}(n_1) \le y \le  \text{arccot}(n_2)$};\\
         1 & \mbox{if $ \text{arccot}(n_2) \le y \le \infty$,}
        \end{array}
 \right.\] 
 and
 \[K_2 (y)=K_2 (-y).\]
 Then:
 \[||K_2 (y)-K_1 (y)||_{\infty}= \epsilon,\]
but it has been shown by making \(n_2-n_1\) small the factors can
differ by an arbitrary large amount in \(L_\infty\) norm, see \cite{cont2}.
\end{rem}

Below, the proposed method of the solution of Problem \(1\) is presented:

\begin{itemize}

\item Approximate, with arbitrary accuracy, the
Wiener-Hopf kernel \(K\) by rational 
functions  \(\tilde{K}\). 

\item Perform the Wiener-Hopf factorisation of the rational kernel
  \(\tilde{K}\) by inspection. The  \(\tilde{K}_+\) will have the form:

\[\frac{(y-z_1) \dots (y-z_n)}{(y-p_1) \dots
  (y-p_m)},\]
where the \(z_i\) (\(p_j\)) are all the zeroes (poles) of  \(\tilde{K}\) which lie in the lower
half-plane. The other factor \(\tilde{K}_-\) will have the remainder
 poles and zeroes in the upper half-plane. 

\item Using Theorem \ref{thm:mul} the error
  \(||K_{\pm}(y)-\tilde{K}_{\pm}(y)||_{p}\) can be calculated.

\end{itemize}

The rest of the paper will concentrate on Problem~2 i.e. the practical
aspect of Problem~1.

 \section{Mapping of the Real Line}
 \label{sec:map}

 
In the previous section it was proved that the approximation of the kernel results in a computable error in the
Wiener-Hopf factors. Thus, to simplify the problem one may wish to
approximate a kernel by another which is easier to factorise e.g. by
rational functions. 
To construct such an approximation the first step is to employ a mapping of the real line to the unit circle or the interval \([-1,1]\). 
Those mappings transform the problem of approximation on
the whole real line to a simpler problem on the unit circle or the
interval. Methods for approximating on the unit circle and the interval will be 
discussed in the next section.

To easily distinguish between functions defined  on the real axis, the unit circle, and the interval
\([-1,1]\) the following notation will be fixed:

\begin{itemize}
\item Functions on the real line will be  denoted by capital letter and
  variable \(y\), for example \(F(y)\); 
\item Functions on the unit circle will be denoted by small case letter and
  variable \(z\), for example \(f(z)\); 
\item Functions on the interval \([-1,1]\) will be denoted by bold capital letters
  and variable \(x\), for example \(\bold{F}(x)\);
  \item The rational approximation for  functions will be denoted by letters with a tilde,
   for example \(\tilde{F}(y).\)
\end{itemize}

It is typical in applications for the kernel to be an
even function  (i.e. \(F(y)=F(-y)\)) and the mapping is simplified in this case.
Therefore it is instructive to consider the mapping for even functions 
before considering the general case.  

\subsection{Even Functions}

Define a  conformal  M\"obious map  \(M(y)\) that maps the real line
to the unit circle as follows \cite{beardon2005algebra}*{Ch 13}:
\[ 1 \mapsto  i, \qquad -1 \mapsto -i,  \qquad \infty \mapsto 1.\]
 Explicitly it is
given by:
\[M(y)=\frac{-1+iy}{1+iy}, \qquad M^{-1}(z)=\frac{i(1+z)}{-1+z}.\]
The next map  is a projection of the unit circle to the
interval \([-1,1]\) by  the Joukowski map (a
conformal map which, incidentally, comes from
aerodynamics \cite{Jouk}). The \emph{Joukowski map} is:
\begin{equation}
  \label{eq:jouk}
  J(z) = \frac{1}{2}(z +z^{-1}).
\end{equation}
Restricted to the unit circle the map returns the real part of \(z\).
This step 
requires the function to be even. Even functions will be
mapped to a function on the unit circle with the property
\(f(z)=f(z^{-1})\). 
Composing these maps together:
\[ JM(y)= \frac{y^2-1}{y^2+1}, \qquad   \qquad M^{-1}J^{-1}(x)=\sqrt{
  \frac{1+x}{1-x}}.\]
This enables the construction of
 a function  \(\bold{F}(x)\) on the interval from a 
given function \(F(y)\) on the real line as follows:
\[\bold{F}(x)=F(M^{-1}J^{-1}(x)). \]
Then the algorithm of the next section can be used to rationally approximate
\(\bold{F}(x)\) on \([-1,1]\) by  \(\tilde{\bold{F}}(x)\). Suppose that
the error in this approximation in the \(L_{\infty}\) is \( \lambda
\). 

Then map the function back to the real line by:
\[\tilde{F}(y)= \tilde{\bold{F}}(JM(y)).\]
Furthermore \(\tilde{F}(y)\) is a rational function on the real line that approximates
\(F(y)\) on the real line with the \(L_{\infty}\)  error at most \(\lambda\).
 
The maps described above are combined with the
Carath\'{e}odory-Fej\'{e}r algorithm in Chebfun to give the code used
for the numerical examples in the next section. 
  
\begin{rem}
If the kernel is odd i.e. \(F(y)=-F(-y)\)  the kernel can be
squared to produce an even kernel. Then the approximation can be
produced as above and square rooted. Note that the explicit multiplicative
Wiener-Hopf factorisation of the square root of a rational function is
done by inspection in the same way as for the rational function.  Also note that any function can be written as a sum of
even and odd function, thus the above extends to a general function. 
\end{rem}

\subsection{General Functions}

For a general kernel the above map can be
modified. The above method contains the following ideas. The real line is
conformally mapped to the circle. Then only half of the circle is
mapped to the interval \([-1, 1]\); however this poses no problem since the other half
is the same due to the function being even. The function is
approximated and mapped back to the unit circle. 
The modification is needed  to ensure that half of the
 function is not discarded. Thus, after mapping the
real line to the circle, apply the map \(S: \text{ } z \to z^2\); this makes the
values run twice as fast and the whole function now fits on the half
circle. Next, as before, map this half circle to the interval,
approximate, and map back. The values that are taken on the half circle
are spread out to the circle by the map \(S^{-1}: \text{ }  z \to
\sqrt{z}\). 
\begin{figure}[htbp]
  \centering
  \includegraphics[scale=0.5,angle=0]{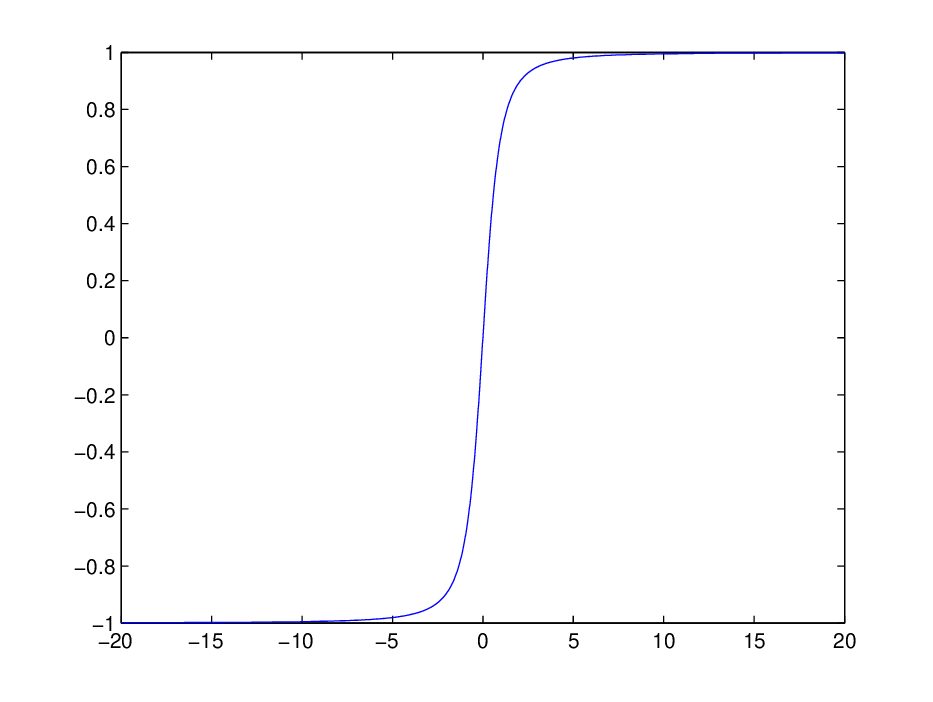}
   \caption{Showing the map JSM(y) of the real line (plotted from \(-20\) to \(20\)) to the interval \([-1,1]\).}
 \label{fig:all_map}
 \end{figure}
Calculating the resulting map gives:
\begin{equation}
\label{eq:all_map}
 JSM(y)= \frac{y}{\sqrt{y^2+1}}, \qquad   \qquad M^{-1}S^{-1}J^{-1}(x)=
  \frac{x}{ \sqrt{1-x^2}},
\end{equation}
 see Figure \ref{fig:all_map}.

\begin{rem}
Before using a map it is beneficial to apply  a rescaling and a
shift (i.e. M\"obious map) to
the real line so that the part of the function which has a fast 
changing gradient fits in the interval \([-5,5]\). This 
part of the real line is well resolved, see  Figure \ref{fig:all_map}.
\end{rem}

\subsection{Associated Orthogonal Rational Functions}

In this section the properties of the map \eqref{eq:all_map} will be 
studied. This will be done by considering the new basis functions which
 are created by this map. A starting point is to choose a basis function 
 for the circle;  a good choice for this is the Fourier basis. Then the 
 Joukowski map will create the new basis for the interval \([-1,1]\) 
  as the image of the Fourier  basis, under the change of 
 coordinates, \cite{Spectral}*{Sec 1.6}. This new basis 
 is the Chebyshev polynomials of the first kind (they are also the Fabe polynomials on \([-1,1]\) 
 up to a multiplicative factor  \cite{Fabe}). Furthermore, the mapping 
 \(M^{-1}S^{-1}J^{-1}(x)\) from \([-1,1]\) to \([-\infty, \infty]\)
  will create a new basis (from Chebyshev polynomials) called 
  \emph{Rational Chebyshev Functions}, \(TB_n\) \cite{BoydTB_n}. 
They are defined as:
\[TB_n(y)=T_n \big(\frac{y}{\sqrt{(y^2+1)}}\big),\]
where the \(T_n\) are the Chebyshev polynomials of the first kind and
\(\frac{y}{\sqrt{(y^2+1)}}\) is the map \(JSM(y)\), see  \eqref{eq:all_map}. Note that, despite the name, \(TB_n\) are not 
all rational:  the
\(TB_{2n+1}\) are rational functions divided by \(\sqrt{(y^2+1)}\) (the \(TB_{2n}\) are all rational function).

A useful  property of  \(TB_n\) is  the shape of the domain of
convergence. This is especially informative when the kernel has no singularities on the whole real
line including infinity. For standard basis functions the domain of convergence has a specific 
shape and the size is determined by the position  
of singularities \cite{BoydTB_n}. For example, Taylor series converge
on a circle and Chebyshev series converge on an ellipse, the sizes of which are dictated by the position of the 
singularities. 
The shape of the domain for \(TB_n\) is the exterior of two circles
which are the  image of lines parallel to the 
real axis under the  M\"obious map \(\frac{i(1+z)}{-1+z}\). This is
 derived from  the corresponding domain of convergence for
Chebyshev  polynomials which in turn comes from the Fourier series  \cite{BoydTB_n}.

So in particular  those base functions are well suited for functions
which have singularities in the bounded part of the complex plane,
like Example \ref{sec:1a}. This compares favourably to the domain of
convergence of Taylor expansion for a function like the one in Example  \ref{sec:1a}. In particular, 
this explains why all methods of approximation which are based on the
Taylor expansion  do not produce good results on the whole real line
(e.g. one point Pad\'{e} approximation).

 For functions with a singularity at infinity the following 
theorem by Boyd is applicable.

\begin{thm} [Convergence of the Fourier series for an algebraically
  decaying or asymptotically  constant function] \cite[Sec 5]{BoydTB_n}
If a function \(u(y)\) is free from singularities on the real axis and has the inverse power expansion
\begin{equation}
\label{eq:asym}
u \sim c_0+c_1/y+c_2/y^2+ \dots
\end{equation}
as \(y \to \infty\) and a similar series as \(y \to -\infty\), then the coefficients of its representation 
as a Fourier series in the new coordinate \(t\) (where \(z=e ^{2 \pi i
  t}\) and \(z=SM(y)\), see  \eqref{eq:all_map}),
\begin{equation}
\label{eq:fourier}
u(y)=\sum_{n=0}^\infty a_n \cos (nt)+ \sum_{n=1}^\infty b_n \sin (nt),
\end{equation}
will have exponential convergence in the sense that \(|a_n|\) and \(|b_n|\) decrease with n faster 
than any finite inverse power of \(n\).

In the case of \eqref{eq:asym} converging then \eqref{eq:fourier} converges geometrically.

In the case of \eqref{eq:asym} being asymptotic but divergent (then
\(u(y)\) will be singular  at infinity but all the derivatives will be
bounded there)  then \eqref{eq:fourier} converges subgeometrically.
\end{thm}

\begin{rem}
The theorem covers precisely the class of functions that arises as
kernels in the  additive and multiplicative Wiener-Hopf factorisation.
In the additive case the kernel must decrease faster than some power
and in the multiplicative case they are asymptotically \(1\).
\end{rem}

The above theorem shows one of the strengths of the \(TB_n\) compared to other base functions on an unbounded interval e.g. the 
Hermite and sinc functions. The latter only converge algebraically for algebraically decaying functions.

\section{Rational Approximation}
\label{sec:rational}

The best rational approximation in
\(L_\infty\) is not practical to use, see \cite{Tref_new}.
This section describes a  near best method of rational approximation
named  AAK (Adamjan-Arov-Kre{\u\i}n) \cite{AAK} in the
 general case and 
Rational Carath\'{e}odory-Fej\'{e}r \cite{Gut_old, Tref_old}
in the case of real valued functions. 
The AAK theory is briefly reviewed at the end. Unfortunately it seems
that there is no  software that performs AAK approximation; only software
for  Rational Carath\'{e}odory-Fej\'{e}r  have been developed
that is 
Chebfun (MATLAB)  \cite{Tref_new}.

\subsection{Real Rational Approximation}

Below a brief summary of the method from \cite{Tref_old_real} is presented. 
Let \(F(x)\) be a real continuous function on \([-1,1]\). For any
natural number \(M\), \(F(x)\) has a partial Chebyshev expansion:
\[F(x)= \sum_{k=0}^Ma_k T_k(x)+G_M(x)= F_M(x)+G_M(x),\]
where
\[a_k= \frac{2}{\pi} \int_{-1}^{1} F(x) T_k(x)
\frac{\text{d}x}{\sqrt{1-x^2}},\]
and \(T_k(x)\) are Chebyshev polynomials of the first kind.

An application of the  Joukowski map (see \eqref{eq:jouk}) to \(F_M(x)\)
produces an associated function on the unit circle:
\[ F_M(x)=\frac{1}{2} f_M(z), \quad \text{where}  \quad
f_M(z)=\sum_{k=-M}^Ma_m z^k.\]
The above is true because of the intimate connection of the
Chebyshev polynomials and Joukowski transform given by:
\[T_k(x)=\frac{1}{2}(z^k +z^{-k}).\]
Define:
\[f^+(z)=\sum_{k=m-n}^Ma_k z^k,\]
where we are seeking a \([m,n]\) rational approximation. 
 Next approximate \(f^+(z)\) on the unit circle by an extended rational
function of the form:
\begin{equation}
  \label{eq:app}
  \tilde{r}(z)=\frac{\sum_{k=- \infty}^m d_k z^k}{\sum_{k=0}^n e_k
  z^k},
\end{equation}
where the numerator is a bounded analytic function in \(|z|>1\)
and the denominator has no zeroes in \(|z|<1\). Call a class of
functions of the above form \emph{\(\tilde{R}_{m,n}\)}. Note that this is not
a typical class to approximate by, but it is the one for which a neat
solution exists.  After the
approximation \eqref{eq:app} is obtained it can be truncated to get a rational approximation. 
To obtain an approximation  in \(\tilde{R}_{m,n}\)  consider  a real
symmetric Hankel matrix:
\begin{equation*}
 \mathbf{X} = \left(
  \begin{array}{cccc}
   a_{m-n+1} & a_{m-n+2} & \ldots & a_m\\
   a_{m-n+2} &  &  &0\\
   \vdots &  &  &\vdots\\
a_m & 0 &  \ldots & 0
  \end{array} \right).
\end{equation*}
Let \(\lambda_i\) be eigenvalues of \(X\) arranged in decreasing order  by  magnitude of absolute
value.  And let \( (u_1, \dots, u_{M+n-m})\) be
an eigenvector for the \(\lambda _{n+1}\) eigenvalue.
Then the following theorem is true:
\begin{thm}[Takagi]
The analytic function \(f^+\) has a unique best approximation
\(\tilde{r}\) on the unit circle \(|z|=1\) in \(\tilde{R}_{m,n}\)
given by:
\[\tilde{r}(z)=b(z) -f^+(z)\]
where \(b\) is the finite Blaschke product

\[b(z)=\lambda_{n+1}z^M  \frac{u_i+ \dots + u_{M+n-m} z^
  {M+n-m}}{u_{M+n-m}+ \dots +u_1 z^{M+n-m-1}}.\]

And the approximation error in \(L_\infty\) norm on the unit circle is
\(|\lambda_{n+1}|\).  
\end{thm}

Once the approximation is obtained, it is mapped back from the
unit circle to the unit interval by the inverse of the Joukowski
transform.  A subsequent truncation gives a near best rational
approximation on the interval.

 This  is a fast and efficient way of computing rational
approximations. Furthermore, it enables to predict the error on the
approximation even before the approximation is computed by considering eigenvalues. 

\subsection{AAK Approximation}

This section presents more general results to the ones covered in
the previous section. These apply to complex valued functions on the
unit circle and are
taken from  \cite{AAK}.

\begin{mydef}
Given a natural number \(n\), define \( H_\infty
^{n} \) a class of bounded functions on the unit circle which can be expressed as:
\[ g(z )=r(z )+ h(z ) ,\]
where \(r(z )\) is a rational function that has no more than \(n\) poles all inside the unit circle 
and \(h(z ) \in H_\infty \).
\end{mydef}

The AAK approximation solves the following problem:

Given a function \(f(z) \in L_\infty\) and \(n\) natural number, find
the best approximation in the \(L_\infty\) norm  from functions in the 
\( H_\infty ^{n} \).

In other words it is required to find:
\[D_n(f)= \inf_{h \in  H_\infty ^{n} } ||f-h||_{\infty}, \]
and a function \( h(z) \in  H_\infty ^{n}\)  (if it exists) such that:
\[ ||f-h||_{\infty} =D_n(f).\]
The solution to the above is presented in the next theorem.

\begin{thm}\cite{AAK}
Let  \(f(z) \in L_\infty\), then  \(D_n(f)=s_n(M)\) where \(M\) is the
Hankel matrix build out of  Fourier coefficients of \(f(z)\) and \(s_n\) is the
\(n^{th}\) singular value of it. Moreover,

\[h(z)=f(z)-s_n \frac{\eta_-(z)}{\xi_+(z)},\]

where \( \{ \xi, \eta \}\)  is a Schmidt pair for a Hankel matrix
\(M\) and where:

\[\xi_+=\sum_{j=1}^\infty \xi_j z^{j-1}, \qquad \eta_-=\sum_{j=1}^\infty \eta_j z^{-j}.\]
\end{thm}

The next section illustrates theory with examples.

\section{Numerical Examples}

In this section the numerical examples are provided to illustrate the
theory (performed in Chebfun, MATLAB).

\subsection{Example 1a}
\label{sec:1a}
\begin{figure}[htbp]
  \centering
  \includegraphics[scale=0.45,angle=0]{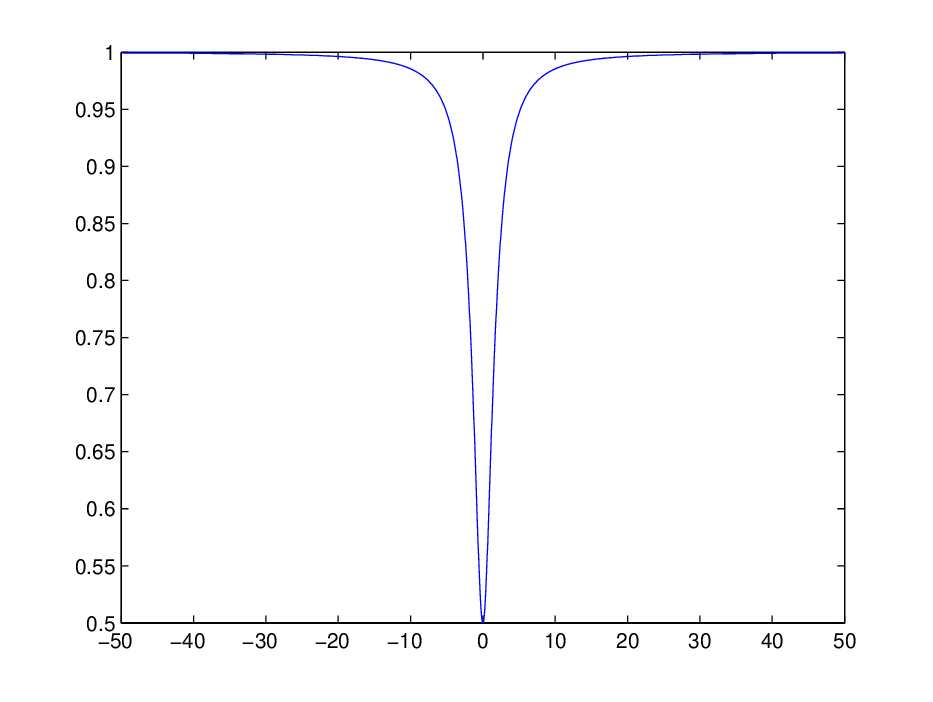}
   \caption{The function to be approximated \(F(y)=\sqrt{ \frac{(y^2+1)}{(y^2+4)}}.\) }
 \label{fig:1ex1}
 \end{figure}
The first example (see Figure \ref{fig:1ex1}) is 
\begin{equation}
  \label{eq:example1}
  F(y)=\sqrt{ \frac{(y^2+1)}{(y^2+k^2)}}.
\end{equation}
Take the finite branch cut from \(i\) to \(ki\) and from  \(-i\) to \(-ki\).
This kernel is closely associated with the matrix kernel factorisation
from problems in acoustics and elasticity and was studied by I. D
Abrahams \cite{Pade}. The case \(k=2\) will be considered in the numerical examples.
The first approach is to approximate this given function on an
interval, a process called domain truncation. On the surface, the
results look  promising (see top
Figure~\ref{fig:3_error}); the error decreases as the degree of  the
numerator polynomial is increased.
\begin{figure}[htbp]
  \centering
  \includegraphics[scale=0.6,angle=0]{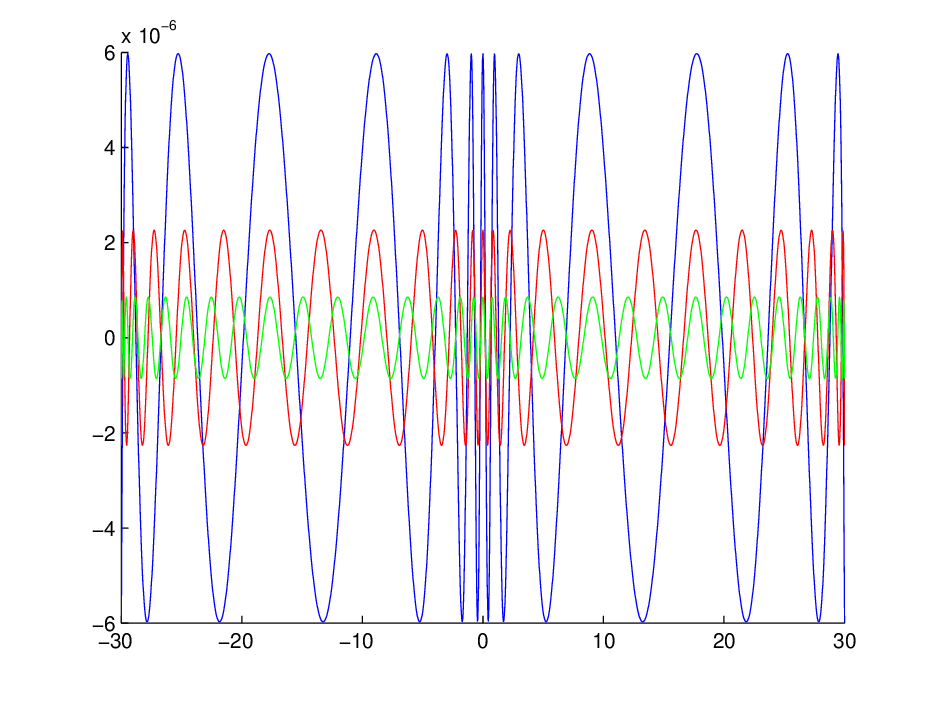}
 \includegraphics[scale=0.6,angle=0]{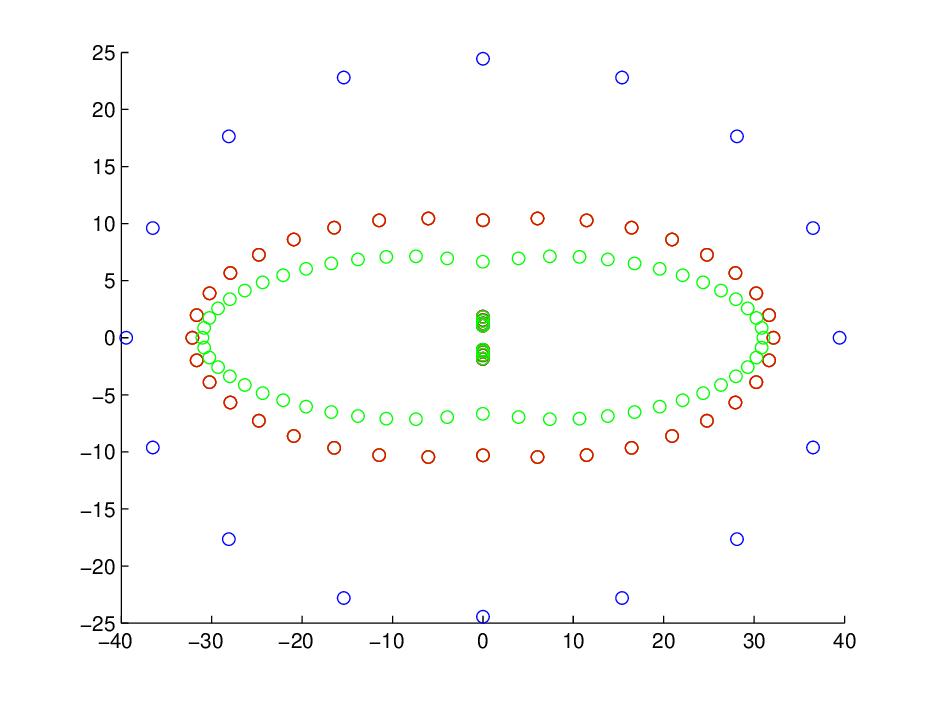}
 \caption{The decreasing  error of [20,4], [40,4] and [60,4]  rational
   approximations (top). The poles and zeroes of [20,4] [40,4] [60,4]
   rational approximations (bottom).}
 \label{fig:3_error}
 \end{figure}
Nevertheless there are  problems, the most obvious being that [20,4]
or any other approximations plotted in top Figure~\ref{fig:3_error} have 
very different behaviours at infinity than \(F(y)\). The error will be
small but only on the given interval and will not be controlled
outside it.
It might be tempting to try to rectify this, since  \(F(y)-1\)
will be zero outside a sufficiently large interval. One might suggest to look
at \([n,m]\) approximations where \(n\) is smaller than \(m\): this
type of function will go to zero at infinity.  But, although the error on
most of the real axis is very small, just outside the interval it
tends to be  greater (and much larger than it is inside the
interval). 
Another feature of this type of approximation is how the poles and
zeroes are positioned. Although some of the poles are positioned on the
branch cuts \([i,2i]\) and \([-i,-2i]\), there are a lot of
singularities introduced elsewhere (see bottom Figure~\ref{fig:3_error}).

 In the next section, the mapping  constructed in Section
 \ref{sec:map} will be
used to overcome the described
problems. This will also result in smaller degree of the polynomials. 

\subsection{Example 1b}
 
The previous example \eqref{eq:example1} is treated with the mapping proposed in Section \ref{sec:map} for even
functions. The function mapped to the interval \([-1,1]\)  becomes:
\[\bold{F}(x)=\sqrt{\frac{2}{5-3x}}.\]
\begin{figure}[htbp]
  \centering
  \includegraphics[scale=0.4,angle=0]{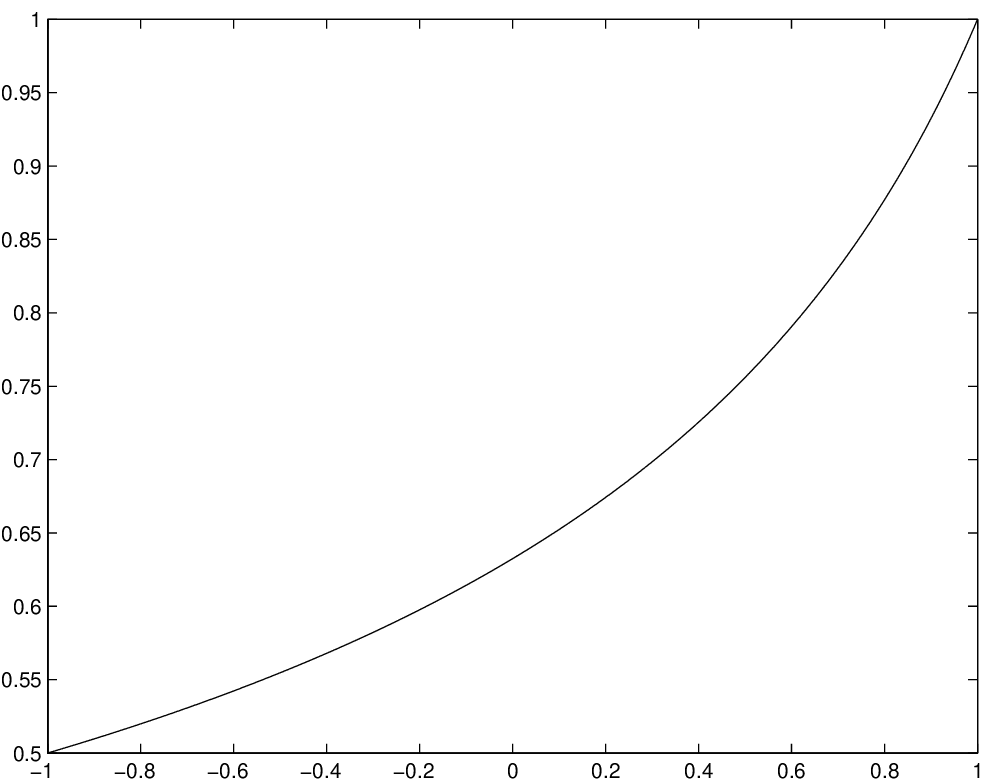}\hfill
   \caption{The new function to be approximated \(\bold{F}(x)=\sqrt{\frac{2}{5-3x}}\).}
 \label{fig:1ex1maped}
 \end{figure}
In Figure \ref{fig:1ex1maped} the new function is plotted;
notice that in this form it is much easier to approximate, since it has a slow changing gradient. 
 This is confirmed by the error curve for the \([4,4]\)
approximation; the error is of order \(10^{-9}\). Furthermore, a
better result is obtained with \([5,5]\) approximation with error of
order \(10^{-12}\) (see Figure~\ref{fig:error10}). 
\begin{figure}[htbp]
  \centering
  \includegraphics[scale=0.54,angle=0]{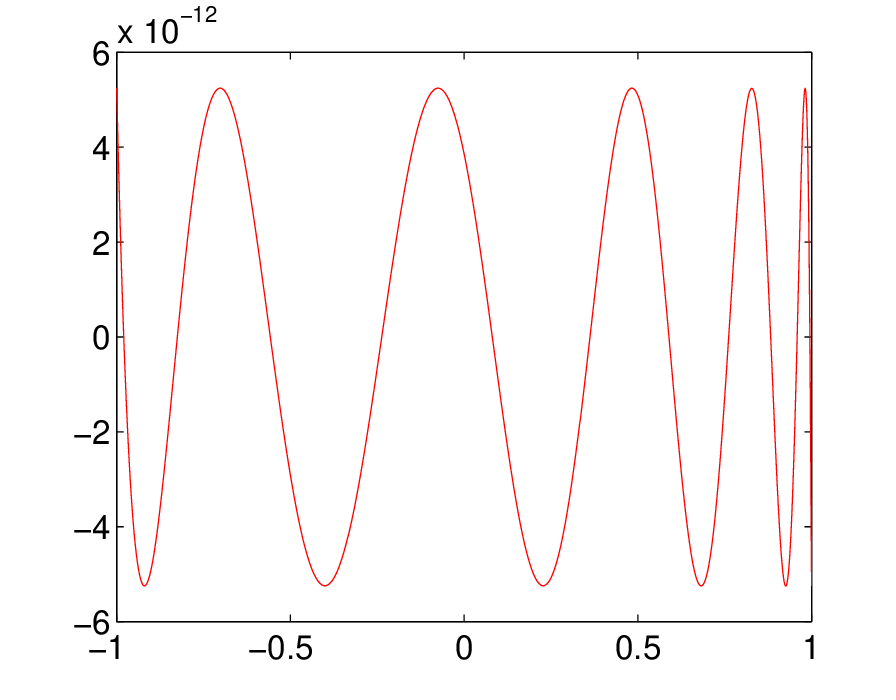}\hfill
   \caption{The  error in approximating  \(\bold{F}(x)\) by [5,5].}
 \label{fig:error10}
 \end{figure}
 \begin{figure}[htbp]
  \centering
  \includegraphics[scale=0.6,angle=0]{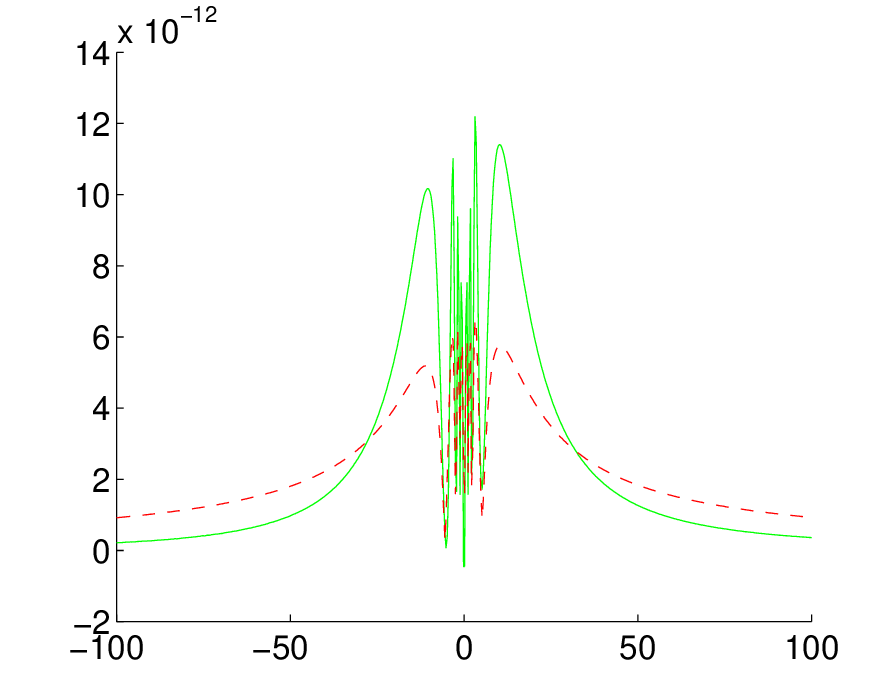}\hfill
   \caption{The  error in approximating  \({F}(y)\) by [10,10] (solid
      curve) and the error after it is split in factors, \(|F_{+}(y)-\tilde{F}_{+}(y)|\) (dotted curve). }
 \label{fig:10}
 \end{figure}
\begin{figure}[htbp]
  \centering
  \includegraphics[scale=0.6,angle=0]{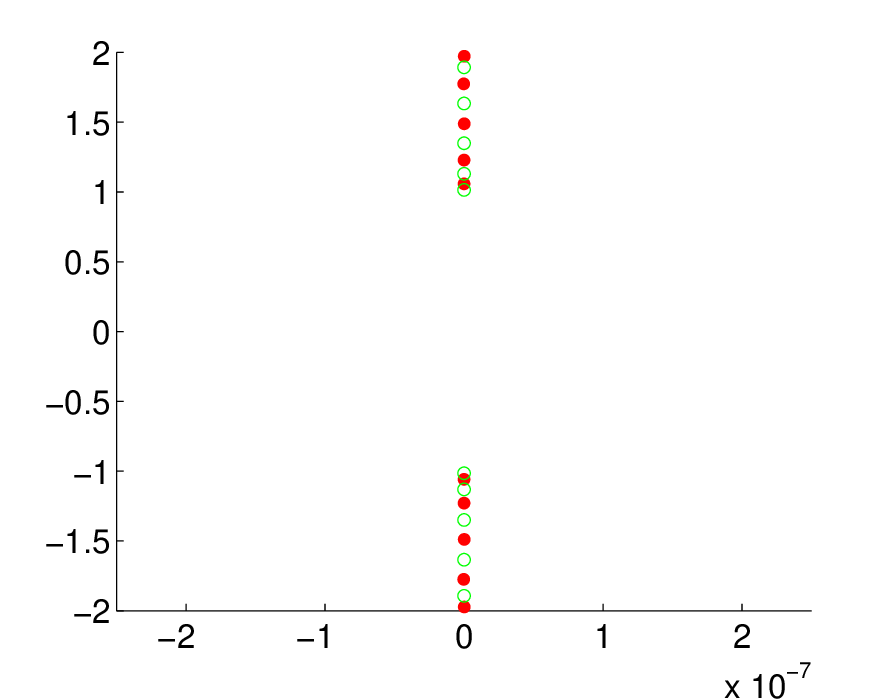}\hfill
   \caption{The  poles (filled in circles) and zeroes (hollow  circles) in approximating  \({F}(y)\) by
     [10,10] .}
 \label{fig:poles}
 \end{figure} 
At this point the approximated function can  be mapped back to the real line
and the \(L_\infty\) norm of the error will stay the same, see
Figure \ref{fig:10} (solid curves). 
 Note that it now 
becomes a [10,10] rational approximation. 
This example was chosen because the Wiener-Hopf  factors can be
easily seen by inspection. They are:
\[F_{\pm}(y)=\sqrt{ \frac{(y\pm i)}{(y\pm ik)}} \qquad
F_{+}(y)=F_-(-y).\] 
The factors of the rational approximation can be easily calculated and
compared to the exact solution above; this is plotted in Figure \ref{fig:10} (dotted curve).
This error is smaller than
the error obtained in \cite{Pade} even with \([30,30]\) for the same function.
 
What is more, the poles and zeroes also lie almost exactly on the branch cuts \([i,2i]\)
and \([-i,-2i]\). The close proximity of zeroes and poles mimic the
behaviour of a branch cut. This indicates that the behaviour in the whole
complex plane  is correct (see Figure~\ref{fig:poles}). Note that the
scale of the \(x\)-axis is \(10^{-7}\).

Bounds developed for real-valued kernels in Section \ref{sec:bound} could be applied to this
example. The \(L_2\) bounds will be used. Note \(c(2)=1\), \(M=1\) and
\(m=1/2\) this gives:

\[||F_{\pm}(y)-\tilde{F}_{\pm}(y)||_2 \le 2
||F(y)-\bar{F}(y)||_2.\]

The \(L_2\) norm is easy to estimate numerically in Chebfun by using
the \emph{norm} command. For the [10,10] rational approximation these calculations give:

\[||F(y)-\tilde{F}(y)||_2=9.5 \times 10^{-10}, \quad  \text{and} \quad
||F_{\pm}(y)-\tilde{F}_{\pm}(y)||_2 =6.5 \times 10^{-10},\] 

which agrees with the bounds above.

\subsection{Example 2}

The function is (Figure~\ref{fig:cf_ex2}):

\[F(y)=\frac{\sqrt{(y^2+1)}}{y} \tanh y.\]

\begin{figure}[htbp]
  \centering
  \includegraphics[scale=0.45,angle=0]{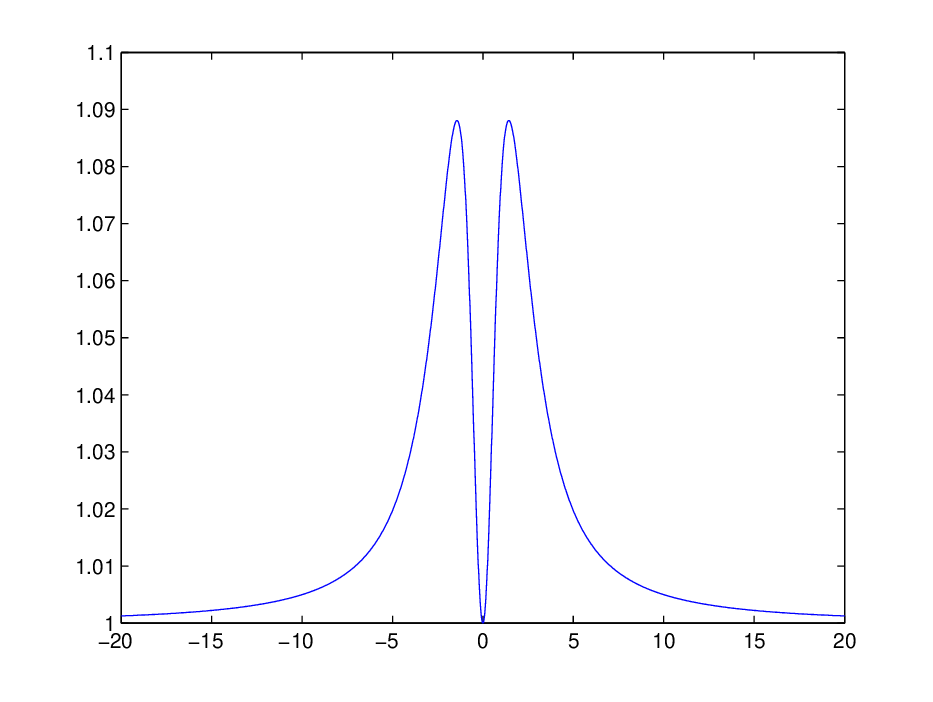}
 \includegraphics[scale=0.45,angle=0]{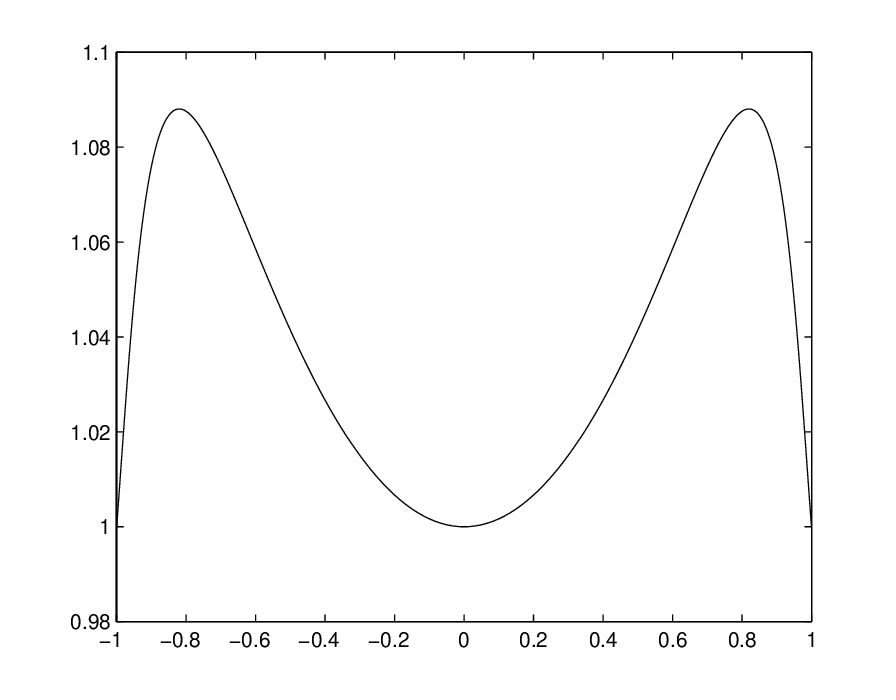}
   \caption{The initial function to be approximated
     \(F(y)=\frac{\sqrt(y^2+1)}{y} \tanh y\) (top) and the  mapped
     function to be approximated \(\bold{F}(x)=\frac{1}{x} \tanh
     \frac{x}{ \sqrt{1-x^2}}\) (bottom).}
 \label{fig:cf_ex2}
 \end{figure}
 This kernel is related to the one considered in Koiter's paper  of \(1954\) \cite{Koiter} and is
typical for electrostatic and slow flow problems. In the original article, rational approximations were considered by choosing the coefficients by hand, accuracy of \(10^{-2}\) was achieved. 
Once again the function becomes easier  to approximate (on the whole real line) once it is mapped to the interval \([-1, 1]\), Figure~\ref{fig:cf_ex2}. The mapped function is (the general mapping is used):
 \[\bold{F}(x)=\frac{1}{x} \tanh  \frac{x}{ \sqrt{1-x^2}}.\]
 This can then be approximated with errors of \(10^{-8}\) (giving the
 graph similar to Figure~\ref{fig:error10}) with
 \([12,12]\) even though
 the point at infinity is singular (it is an accumulation point of a
 sequence of poles). Interestingly, the position of zeroes and poles (see
 Figure \ref{fig:cf_ex2_roots})   is
 very similar to  that found in \cite{Pade}*{Fig 8} with a
 different method of approximating. 
 \begin{figure}[htbp]
  \centering
  \includegraphics[scale=0.6,angle=0]{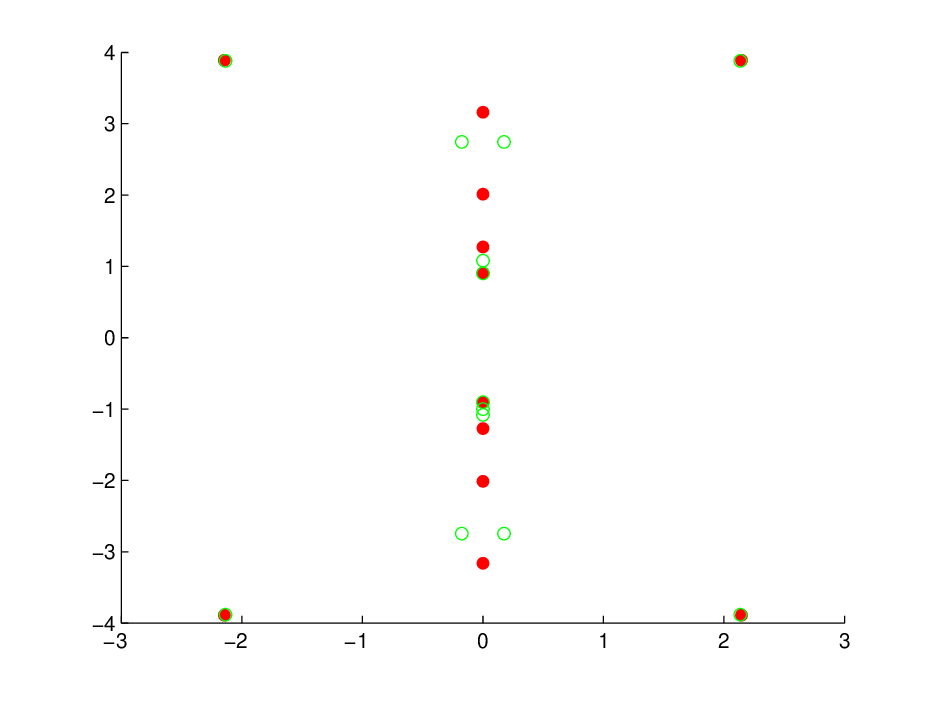}
   \caption{The  poles (filled in circles) and zeroes (hollow  circles) in approximating  \({F}(y)\) by
     [12,12].}
 \label{fig:cf_ex2_roots}
 \end{figure}
This kernel has exact factorisation which is given by:
\[K_{\pm} (y)= \frac{e^{-i \pi/4}\, \Gamma (1/2-iy/ \pi)}
{ \sqrt{\pi}\, \Gamma (1-iy/ \pi)} (i\pm\alpha)^{1/2} .\]
 But MATLAB does not have an inbuilt complex \( \Gamma\) function
 hence the results of this approximation cannot be compared to the
 exact factorisation. Even if there was an inbuilt complex \( \Gamma\)
 function to compare it to the approximate solution the inbuilt
 function would need to have very high accuracy. This demonstrate that if
 the numerical values of the solution is needed the approximate
 solution is as good as the exact solution.

\section{Conclusion}

This paper demonstrates with use of examples the method of approximate
solution of the Wiener-Hopf equation. The bounds presented allow to
predict the error in the approximation. The degrees of polynomials in
the rational approximation are small which allow to simplify the
initial problem significantly. The next project would be to generalise
the approach to the matrix Wiener-Hopf problem.

\section{Acknowledgements}

I am grateful to  Prof Nigel Peake for suggesting this project and for
support along the way. Anonymous referees made several useful comments
which helped to improve this paper.

{\small

}
\end{document}